\documentclass[10pt]{amsart}

\usepackage{latexsym}
\usepackage{amsmath}
\usepackage{amsfonts}
\usepackage{amssymb}

\newtheorem{theorem}{Theorem}[section]
\newtheorem{lemma}[theorem]{Lemma}
\newtheorem{corollary}[theorem]{Corollary}
\newtheorem{proposition}[theorem]{Proposition}

\newtheorem{sublemma}{}[theorem]
\newtheorem{conjecture}[theorem]{Conjecture}

\theoremstyle{definition}

\theoremstyle{remark}

\numberwithin{equation}{section}

\newcommand{\ba}{\backslash}

\usepackage{graphicx}
\DeclareGraphicsExtensions{.pdf,.png,.jpg}

\usepackage{enumerate}

\usepackage[usenames]{color}

\begin{document}

\title[$2$-polymatroid minors]{A note on the connectivity of $2$-polymatroid minors}

\author{Zachary Gershkoff}
\address{Mathematics Department\\
Louisiana State University\\
Baton Rouge, Louisiana}
\email{zgersh2@math.lsu.edu}

\author{James Oxley}
\address{Mathematics Department\\
Louisiana State University\\
Baton Rouge, Louisiana}
\email{oxley@math.lsu.edu}

\date{\today}

\begin{abstract}
Brylawski and Seymour independently proved that if $M$ is a connected matroid with a connected minor $N$, and $e \in E(M) - E(N)$, then $M \backslash e$ or $M / e$ is connected having $N$ as a minor. This paper proves an analogous but somewhat weaker result for $2$-polymatroids. Specifically, if $M$ is a connected $2$-polymatroid with a proper connected minor $N$, then there is an element $e$ of $E(M) - E(N)$ such that $M \backslash e$ or $M / e$ is connected having $N$ as a minor. 
We also consider what can be said  about the uniqueness of the way in which the elements of $E(M) - E(N)$ can be removed so that connectedness is always maintained.
\end{abstract}

\maketitle

\section{Introduction}
\label{Introduction}

Tutte \cite{wtt} proved that, whenever $e$ is an element of a  connected matroid $M$, at least one   of $M \backslash e$ and $M / e$ is connected. Brylawski \cite{bry} and Seymour \cite{sey} independently extended this theorem by showing that if $N$ is a connected minor of $M$ and $e$ is in $E(M) - E(N)$, then   $M \ba e$ or $M / e$ is connected having $N$ as a minor. In this paper, we prove a similar result for $2$-polymatroids.

For a positive integer $k$, 
a  {\em $k$-polymatroid}  $M$ is a pair 
$(E,r)$ consisting of a finite {\it ground set} $E$  and a {\it rank function} $r$,  
  from the power set of $E$ into the integers, satisfying the 
following  conditions:

\begin{itemize}
\item[(i)] $r(\emptyset) = 0$;
\item[(ii)] if $X \subseteq Y \subseteq E$, then $r(X) \le r(Y)$;
\item[(iii)] if $X$ and $Y$ are subsets of $E$, then $r(X) + r(Y) \ge r(X \cup Y) + r(X \cap Y)$; 
and
\item[(iv)] $r(\{e\})\leq k$ for all $e\in E$.
\end{itemize}

A matroid is just a 1-polymatroid, so every matroid is a $2$-polymatroid. We call $M$ a {\em polymatroid} if $M$ is a $k$-polymatroid for some $k$.    Our focus here will be mainly on $2$-polymatroids. Elements of a polymatroid of ranks $0$, $1$, and $2$ are called {\it loops}, {\it points}, and {\it lines}, respectively. Non-loop elements $p$ and $q$ are {\it parallel} if $r(\{p,q\}) = r(\{p\}) = r(\{q\})$.

Many matroid  concepts that are stated in terms of the rank function can be extended to polymatroids. In particular, for a polymatroid $M= (E,r)$ and a subset $T$  of $E$, the {\it deletion} $M\ba T$ and the {\it contraction} $M/T$ of $T$ from $M$ are the polymatroids with ground set $E - T$ and rank functions $r_{M\ba T}$ and $r_{M/T}$ where $r_{M\ba T}(X) = r(X)$ and $r_{M/ T}(X) = r(X\cup T) - r(T)$ for all subsets $X$ of $E- T$. A {\it minor} of $M$ is any polymatroid that can be obtained from $M$ by a sequence of deletions and contractions. A 
 polymatroid $M$ is {\it connected}, or equivalently {\it $2$-connected},  if there is no non-empty proper subset $X$ of its ground set $E$ such that 
 $r(X) + r(E - X) = r(E)$.  We sometimes use   $E(M)$ and $r_M$ to denote the ground set and rank function of $M$.
 
The following is the main result of the paper. 

\begin{theorem}
\label{mainone}
Let $M$ be a connected $2$-polymatroid and let $N$ be a connected minor of $M$. When $N \neq M$, there is an element $e$ of $E(M) - E(N)$ such that $M \ba e$ or $M / e$ is connected having $N$ as a minor.
\end{theorem}

Unlike in the matroid case, it is not true that, for every element $e$ of $E(M) - E(N)$,  at least one of $M \ba e$ and $M / e$ is connected having $N$ as a minor.  For example, 
 let $E(M) = \{x,y,z\}$ where $x$, $y$, and $z$ are lines,  $r(\{x,y\}) = r(\{y,z\})= 3$, and $r(\{x,z\}) = 4$. Let $N$ be the 2-polymatroid consisting of a single line $z$.  
Then both $M \ba y$ and $M / y$ are disconnected, as the former consists of  two lines in rank $4$, and the latter is isomorphic to the matroid $U_{2,2}$.

Theorem~\ref{mainone} will be proved in Section~\ref{splitter}.  The next section includes a number of preliminaries needed for this proof. In Section~\ref{uniq}, we consider what can be said about the uniqueness of the element $e$. 

\section{Preliminaries}
\label{prelim}

Our matroid terminology   follows Oxley~\cite{oxbook}. Indeed, much of the notation from matroid theory carries over to polymatroids. For instance, when $M$ is the polymatroid $(E,r)$ and $T \subseteq E$, the deletion $M\ba (E-T)$ is also denoted by $M|T$. Moreover, we frequently write $r(M)$ for $r(E)$.  A subset $S$ of $E$ {\it spans} a subset $T$ if $r(S\cup T) = r(S)$. A {\it component} of $M$ is a maximal non-empty subset $X$ of $E$ such that $M|X$ is connected. 
As for matroids,  the {\it connectivity function} $\lambda_M$ or $\lambda$ of $M$ is defined for all subsets $X$ of $E(M)$ by $\lambda_M(X) = r(X) + r(E-X) - r(E)$. For a positive integer $j$ and a subset $Z$ of $E(M)$, we call $Z$ and $(Z,E(M) - Z)$ {\it $j$-separating} if $\lambda_M(Z) < j$. 

The {\it local connectivity} $\sqcap(X,Y)$  between subsets $X$ and $Y$ of $E$  is given by $\sqcap(X,Y)=r(X) + r(Y) - r(X \cup Y)$.  Thus 
$\sqcap(X,E-X) = \lambda(X)$. 
The following   useful   results  for local connectivity and connectivity are proved for matroids  in \cite[Lemmas~8.2.3 and 8.2.4]{oxbook}; the proofs there extend to polymatroids. 


\begin{lemma}
\label{8.2.3} Let $(E,r)$ be a polymatroid   and
let $X_1, X_2, Y_1$, and $Y_2$ be subsets of $E$ with $Y_1 \subseteq X_1$ and $Y_2 \subseteq X_2$. 
Then 
$$\sqcap(Y_1,Y_2) \le \sqcap(X_1,X_2).$$
\end{lemma}

\begin{lemma}
\label{8.2.4} Let $(E,r)$ be a polymatroid $M$ and
let $X,C$, and $D$ be disjoint subsets of $E$. 
Then 
$$\lambda_{M\ba D/C}(X) \le \lambda_M(X).$$
Moreover, equality holds if and only if 
$$r(X \cup C) = r(X) + r(C)$$ and 
$$r(E - X) + r(E-D) = r(E) + r(E- (X \cup D)).$$
\end{lemma} 

Next we note  a useful consequence of Lemma~\ref{8.2.3}.

\begin{corollary}
\label{xyconn} Let $X$ and $Y$ be sets in a polymatroid $M$ such that $X\cap Y \neq \emptyset$ and both $M|X$ and $M|Y$ are connected. 
Then $M|(X \cup Y)$ is connected.
\end{corollary}

\begin{proof}
Suppose that $M|(X \cup Y)$ is disconnected, and let $Z$ be a component of it. Let $W = (X \cup Y) - Z$. By Lemma~\ref{8.2.3},  
$\sqcap(Z \cap X,W\cap X) \le \sqcap(Z,W) = 0$. As $M|X$ is connected, $Z\cap X$ or $W\cap X$ is empty. By symmetry, $Z\cap Y$ or $W\cap Y$ is empty. As neither $Z$ nor $W$ is empty, we may assume that both $Z\cap X$ and $W\cap Y$ are empty. It follows that $X\cap Y$   is empty, a contradiction.
\end{proof}

The following generalization of a matroid result was noted in \cite[Lemma 3.12(ii)]{osw2}. 

\begin{lemma}
\label{oswrules} Let $A, B,$  and $C$  be subsets of the ground set of a polymatroid. Then 
 $$\sqcap(A \cup B, C) + \sqcap(A,B) = 
\sqcap(A \cup C, B) + \sqcap(A,C).$$
\end{lemma}


We omit the proof of the next result, which follows easily from using submodularity of the rank function.

\begin{lemma}
\label{newt}
If $\lambda(Z) = 0$ in a polymatroid $M$, then $M\ba Z = M/Z$. 
\end{lemma}

As noted in \cite[p.409]{oxbook}, with every $2$-polymatroid $M$, we can associate a matroid   as follows. 
Let $L$ be the set of lines of $M$. For each $\ell$ in $L$,   freely  add  two points  to $\ell$ letting $M^+$ be the resulting $2$-polymatroid.   Then $M'$, the {\it natural matroid derived from $M$}, is  $M^+ \ba L$.  Oxley, Semple, and Whittle \cite[Lemma 3.3]{osw2} noted the following straightforward result.

\begin{lemma}\label{natural}
Let $M$ be a $2$-polymatroid with $|E(M)| \geq 2$ and let $M'$ be the   natural matroid derived from $M$. Then $M$ is connected if and only if $M'$ is connected.
\end{lemma}

The proof of our main theorem will use the operations of  parallel connection and $2$-sum of polymatroids as introduced by Mat\'{u}\v{s}~\cite{fm} and Hall~\cite{hall}. 
For a positive integer $k$, let $M_1$ and $M_2$ be $k$-polymatroids $(E_1,r_1)$ and $(E_2,r_2)$. Suppose first that $E_1 \cap E_2 = \emptyset$. 
The {\it direct sum} $M_1 \oplus M_2$ of $M_1$ and $M_2$ is the $k$-polymatroid $(E_1 \cup E_2,r)$ where, for all subsets $A$ of $E_1 \cup E_2$, we have $r(A) = r(A\cap E_1) + r(A \cap E_2)$. 
Clearly a $2$-polymatroid is  connected if and only if it cannot be written as the direct sum of two non-empty $2$-polymatroids. 
Now suppose that  $E_1 \cap E_2 = \{p\}$ and $r_1(\{p\}) = r_2(\{p\})$. Let $P(M_1,M_2)$ be $(E_1 \cup E_2, r)$ where $r$ is defined for all subsets $A$ of $E_1 \cup E_2$ by 
$$r(A) = \min\{r_1(A \cap E_1) + r_2(A\cap E_2), r_1((A \cap E_1)\cup p) + r_2((A \cap E_2)\cup p) - r_1(\{p\})\}.$$
 Hall~\cite{hall} notes that it is routine to check that $P(M_1,M_2)$ is a $k$-polymatroid. We call it the {\it parallel connection} of $M_1$ and $M_2$ with respect to the {\it basepoint} $p$. When $M_1$ and $M_2$ are both matroids, this definition coincides with the usual definition of the parallel connection of matroids.  
 
Now let $M_1$ and $M_2$ be $2$-polymatroids having at least two elements. Suppose that $E(M_1) \cap E(M_2) = \{p\}$, 
that neither $\lambda_{M_1}(\{p\})$ nor $\lambda_{M_2}(\{p\})$ is $0$, and that 
$r_1(\{p\}) = r_2(\{p\}) = 1$. We define the {\it $2$-sum}, $M_1 \oplus_2 M_2$, of $M_1$ and $M_2$ to be $P(M_1,M_2)\ba p$. This definition~\cite{osw2} extends Hall's definition  since the latter requires each of $M_1$ and $M_2$ to  have at least three elements. 
Weakening that requirement does not alter the validity of Hall's proof of the following result~\cite[Proposition 3.6]{hall}.

\begin{proposition} 
\label{dennis3.6}
Let $M$ be a $2$-polymatroid $(E,r)$ having a partition $(X_1,X_2)$ of $E$ such that $r(X_1) + r(X_2) = r(E) + 1$. Then there are $2$-polymatroids $M_1$ and $M_2$ with ground sets $X_1 \cup p$ and $X_2 \cup p$, where $p$ is a new element, such that 
$M = P(M_1,M_2)\ba p$. In particular, for all $A \subseteq X_1 \cup p$,
\begin{equation*}
r_1(A) = 
\begin{cases} 
r(A), & \text{if   $p \not\in A$;}\\
r((A-p) \cup X_2) - r(X_2) + 1, & \text{if $p \in A$.}
\end{cases}
\end{equation*}
\end{proposition}

The following was shown by Hall~\cite[Corollary 3.5]{hall}. 

\begin{proposition}
\label{connconn} 
Let $M_1$ and $M_2$ be $2$-polymatroids $(E_1,r_1)$ and $(E_2,r_2)$ where $E_1 \cap E_2 = \{p\}$. Suppose $r_1(\{p\}) = r_2(\{p\}) = 1$ and each of $M_1$ and $M_2$ has at least two elements. Then the following are equivalent.
\begin{itemize}
\item[(i)] $M_1$ and $M_2$ are both $2$-connected; 
\item[(ii)] $M_1 \oplus_2 M_2$ is $2$-connected; and 
\item[(iii)] $P(M_1,M_2)$ is $2$-connected. 
\end{itemize}
\end{proposition}

The next theorem, a special case of    a result of Hall~\cite[Theorem 4.3]{hall}, will play a crucial role in the proof of our main theorem.  

\begin{theorem}
\label{dennismain}
Every connected $2$-polymatroid $M$ having at least two elements has distinct elements $x$ and $y$ such that each of 
$\{M\ba x,M/x\}$ and $\{M\ba y,M/y\}$ contains a connected $2$-polymatroid.
\end{theorem}

The proof of our main theorem will also use the next two results, which are specializations of \cite[Theorem~3.1]{oxwh} and \cite[Lemma~46]{osw2}, respectively.

\begin{lemma}\label{new1}
Let $A$ be a subset of the ground set of a connected polymatroid $M$. If $\lambda(A) < 2 \min\{ \lambda(X) : \emptyset \neq X \subsetneqq E(M)\}$, then $M \ba A$ or $M / A$ is connected.
\end{lemma}

\begin{lemma}\label{new2}
Suppose that the $2$-polymatroid $M$ is the $2$-sum of polymatroids $M_1$ and $M_2$ that have ground sets $E_1 \cup p$ and $E_2 \cup p$, respectively. For $q$ in $E_1$,

\begin{enumerate}[(i)]
\item \label{new2i} if $\sqcap(E_1 - q, E_2) = 1$, then $M \ba q = (M_1 \ba q) \oplus_2 M_2$; and
\item \label{new2ii} if $\sqcap( \{q\}, E_2 ) = 0$, then $M / q = (M_1 / q) \oplus_2 M_2$.
\end{enumerate}
\end{lemma}

The following lemma holds for polymatroids in general and will be useful in Section~\ref{uniq}.

\begin{lemma}\label{lconnectivity}
Let $M$ be a connected polymatroid and let $M / e$ be disconnected. If $Z$ is a component of $M / e$, then $\sqcap(Z, \{e\}) > 0$.
\end{lemma}

%

\begin{proof}
Let $Y = E(M) - (Z \cup e)$. Then
\begin{equation}
\label{*12}
r_{M / e} (Z) + r_{M / e} (Y) = r(M / e) = r(M) - r_M(\{e\}).
\end{equation}
Moreover, since $M$ is connected,
\begin{equation}
\label{*13}
r_{M} (Z) + r_{M} (Y \cup e) > r(M).
\end{equation}
By the definition of local connectivity, $r_M(Y) + r_M(\{e\}) - \sqcap(Y,\{e\}) = r(Y \cup e)$, so we can rewrite (\ref{*13}) as 
\begin{equation}
\label{*15}
r_{M} (Z) + r_{M} (Y) > r(M)  - r_M(\{e\}) + \sqcap(Y,\{e\}) .
\end{equation}
By subtracting (\ref{*12}) from (\ref{*15}), we obtain
\begin{equation}
\label{*14}
(r_M (Z) - r_{M / e} (Z)) + (r_M (Y) - r_{M / e} (Y)) > \sqcap(Y,\{e\}).
\end{equation}
The  differences on the left-hand side can be rewritten as  local connectivities. Thus $\sqcap(Z,\{e\}) + \sqcap(Y,\{e\}) > \sqcap(Y,\{e\})$, so $\sqcap(Z,\{e\}) > 0$.
\end{proof}

\section{A Splitter Theorem for Connected $2$-polymatroids}
\label{splitter}

This section is devoted to proving  the main result of the paper.

\begin{proof}[Proof of Theorem~\ref{mainone}.]
Assume that the theorem fails. Then it follows from Theorem~\ref{dennismain} that $N$ is non-empty. 
Hence, as $M$ is connected, it has no loops. Next we note the following.

\setcounter{theorem}{1}

\begin{sublemma}
\label{sublemma-1}
If $x$ is an element of $M$ such that both $M\backslash x$ and $M/x$ have $N$ as a minor, then $x$ is a line of $M$.
\end{sublemma}

Clearly neither $M\ba x$ nor $M/x$ is connected. It follows by Lemma~\ref{new1} that   $r_M(\{x\}) \neq 1$. Hence $x$ is a line of $M$.

Take $e$ in $E(M) - E(N)$. Then, for some $M_0$ in $\{M\backslash e, M/e\}$, the $2$-polymatroid $M_0$ has $N$ as a minor.  By assumption,  $M_0$ is not connected. Take an element $f$ of a  component   of $M_0$ that avoids $E(N)$. Then both 
$M\backslash f$ and $M/f$ have $N$ as a minor. Thus, by \ref{sublemma-1},  $f$ is a line of $M$. Moreover, 

\begin{sublemma}
\label{sublemma0}
$r(M\ba f) = r(M)$.
\end{sublemma}
To see this, suppose $r(E - f) < r(E)$. Then, as $M$ is connected, $r(E - f) = r(E) -1$. Let $M'$ be the natural matroid derived from $M$ and let $f_1$ and $f_2$ be the points of $M'$ corresponding to $f$. Then $M'\ba f_1$ has $f_2$ as a coloop, so $M'/f_1$ is connected. Now $M/f$ is disconnected, so, by Lemma~\ref{natural}, $M'/f_1,f_2$ is disconnected. Therefore, $M'/f_1\ba f_2$ is connected. But $M' \ba f_2$ has $f_1$ as a coloop, so $M'\ba f_2/f_1 = M'\ba f_2\ba f_1$. As this matroid is connected,  by Lemma~\ref{natural},  $M\ba f$ is too; a contradiction.

\begin{sublemma}\label{sublemma1}
Let $K$ be a component of $M/f$. Then $M|(K \cup f)$ is connected. 
\end{sublemma}

Suppose $M|(K \cup f)$ is disconnected. Then $K$ is the disjoint union of sets $X$ and $Y$ such that $r(X \cup f) + r(Y) = r(K \cup f)$. As $Y$ is $1$-separating in $M|(K \cup f)$, it is $1$-separating in $(M|(K \cup f))/f$, that is, in $(M/f)|K$. But $K$ is a component of the last matroid, so $K = Y$. Thus $X = \emptyset$, so $r(K \cup f) = r(K) + r(\{f\})$. It follows that $K$ is $1$-separating in $M$; a contradiction.  Hence \ref{sublemma1} holds.



Now let $F$ be  a component of $M/f$ that avoids $E(N)$. By \ref{sublemma-1}, every element of $F$ is a line in $M$. Let $G= E(M) - f - F$. By \ref{sublemma1}, $M|(F \cup f)$ is connected.  
Next we show  the following.

\begin{sublemma}\label{den1}
There is a line $g$ in $F$ such that $(M|(F \cup f)) \ba g$ or $(M|(F \cup f)) / g$ is connected. Moreover, $\sqcap(\{f\}, \{g\})  < 2$.
\end{sublemma}

By Theorem~\ref{dennismain}, $F$ contains an element $g$ such that $(M|(F \cup f)) \ba g$ or $(M|(F \cup f)) / g$ is connected.
As $g$ is in $F$,  we see that $g$ is  a line. Since the theorem fails,   $M \ba g$   is not connected, so $\sqcap(\{f\}, \{g\})  < 2$. Thus \ref{den1} holds.


\begin{sublemma}\label{sublemma2}
$M|(G \cup f)$ is connected, $(M | (F \cup f)) / g$ is connected, and  $r(\{f, g\}) = 3$.
\end{sublemma}

To see this, first note that, by \ref{sublemma1},   $M|(K \cup f)$ is connected for each component $K$ of $M/f$.  Then, by Corollary~\ref{xyconn}, 
 $M|(G \cup f)$ is connected. 
The same argument shows   that  $(M|(F \cup f))\ba g$ is disconnected for if it is connected, then so is $M\ba g$; a contradiction.  
Thus, by \ref{den1}, $(M | (F \cup f)) / g$ is connected. 

Now suppose that  $\sqcap_M(\{f\},\{g\}) = 0$. Then, as $\sqcap_{M/f}(G,\{g\}) = 0$, one easily checks that $\sqcap_M(G \cup f,\{g\}) = 0$. Hence $M|(G \cup f) = (M|(G \cup f \cup g))/ g = (M/g)|(G \cup f)$. Thus, as $(M|(F \cup f))/ g$ and $M|(G \cup f)$ are connected and both contain $f$, Corollary~\ref{xyconn} implies that $M/ g$ is connected; a contradiction.   Hence \ref{sublemma2} holds.

Recall that $f$ is a line of $M$ such that $M \ba f$ and $M / f$ are disconnected. Moreover, $F$ is a component of $M/f$ and $E(N)\subseteq  G =E(M) - f - F$. Let $A$ be a component of $M \ba f$  avoiding $E(N)$ and let $B = E(M) -f - A$.  
The next two observations follow because $M$ is connected.

\begin{sublemma}\label{nospan}
Neither $A$ nor $B$ spans $f$. 
\end{sublemma}

\begin{sublemma}\label{notskew}
$r(G \cup f) < r(G) + 2$ and $r(F \cup f) < r(F) + 2$.
\end{sublemma}


Next we show the following.

\begin{sublemma}\label{notempty}
At least one of $A \cap G$, $A \cap F$, $B \cap F$, and $B \cap G$ is empty.
\end{sublemma}

Suppose that all four intersections   are non-empty. By \ref{sublemma0}, $r(E - f) = r(E)$. Thus $r(A) + r(B) = r(E)$ and 
$r(F \cup f) + r(G \cup f) = r(E) + 2$. Adding these two equations and applying submodularity to the left-hand side gives 
$$r(A \cup F \cup f) + r(A \cap F) + r(B \cup G \cup f) + r(B \cap G) \le 2r(E) + 2,$$ so 
\begin{equation}
\label{eq2}
[r(A \cup F \cup f) + r(B  \cap G)] + [r(B \cup G \cup f) + r(A \cap F)] \le 2r(E) + 2.
\end{equation}

As $(A \cup F \cup f, B\cap G)$ and $(B\cup G \cup f, A\cap F)$ are partitions of $E(M)$, we deduce, since $M$ is connected, that equality holds in (\ref{eq2}). Hence the two specified partitions are $2$-separating in $M$. By symmetry, so are $(A \cup G \cup f, B\cap F)$ and $(B\cup F \cup f, A\cap G)$. By   Propositions~\ref{dennis3.6} and  \ref{connconn},  $M$ can be written as the 2-sum with basepoint $p_{AF}$ of two connected $2$-polymatroids, one with ground set $(A\cap F) \cup p_{AF}$ and the other, $Q_0$, with ground set $(E(M) - (A \cap F)) \cup p_{AF}$. By arguing in terms of the natural matroid derived from $M$, it is straightforward to check that, in $Q_0$, each of  $A\cap G, B\cap F$, and $B\cap G$ is $2$-separating.  Hence we can decompose $Q_0$ as a $2$-sum of two connected 2-polymatroids one with ground set $(A\cap G) \cup p_{AG}$. Repeating this process twice more, we obtain a connected 2-polymatroid $Q$ with ground set $\{f,p_{AF},p_{AG}, p_{BF},p_{BG}\}$ where $M$ is obtained from $Q$ by attaching, via 2-sums, connected $2$-polymatroids with ground sets 
$(A\cap F) \cup p_{AF}$, $(A\cap G) \cup p_{AG}$, $(B\cap F) \cup p_{BF}$, and $(B\cap G) \cup p_{BG}$.

As $M|A$ is connected, Proposition~\ref{connconn} implies that  $p_{AG}$ and $p_{AF}$ are parallel in $Q$.  Since $(M / f)|F$  is connected,    $p_{BF}$ and $p_{AF}$ are   parallel in $Q / f$. But  $p_{AG}$ and $p_{AF}$ are also  parallel in $Q / f$ unless they are loops. In the exceptional case,     $A \cap F$ contains a component of $M/f$; a contradiction. We deduce that the component of $M / f$ containing $F$   also contains $A \cap G$; a contradiction.  Thus \ref{notempty} holds. 

By \ref{notempty}, $A$ or $B$ is contained in   $F$ or $G$, and   $F$ or $G$ is contained in  $A$ or $B$. We know that $B \cap G$ is non-empty because it contains $E(N)$.

Suppose both $F$ and $G$ span $f$. Then    $A$ or $B$ spans $f$; a contradiction to   \ref{nospan}. By \ref{notskew}, there are two remaining cases to consider:

\begin{itemize}
\item[(i)]  $r(F \cup f) = r(F) + 1$; and
\item[(ii)] $r(F \cup f) = r(F)$ and $r(G \cup f) = r(G) + 1$.
\end{itemize}

By \ref{sublemma2},   $(M | (F \cup f) )/ g$ is connected  and $\sqcap(\{f\},\{g\}) = 1$. Thus $r(\{f,g\}) = 3$. Assume (i) holds. Then 
 $\sqcap_{M/g}(F-g,\{f\}) = r(F) + r(\{f,g\}) - r(F \cup f) - r(\{g\})= 0$. Thus $\{f\}$ is  a component of $(M|(F \cup f))/g$. As the last polymatroid is connected, we deduce that $F = \{g\}$. Thus $M\ba g = M|(G \cup f)$ so, by \ref{sublemma2}, $M\ba g$ is connected; a contradiction.

We now know that (ii) holds. As neither $A$ nor $B$ spans $f$,  neither has $F$ as a subset. Thus both $A \cap F$ and $B \cap F$ are  non-empty. As $B \cap G$   is non-empty,  \ref{notempty} implies that $A \cap G$ is empty. Then $G \subseteq B$. But 
$r(G \cup f) = r(G) + 1$. Therefore $r(B \cup f) \le r(B) + 1$. Since $r(B \cup f) \neq r(B)$, it follows that $r(B \cup f) = r(B) + 1$.

We have $r(A) + r(B) = r(M\ba f)$, and, by \ref{sublemma0}, $r(M\ba f) = r(M)$. As $r(B \cup f) = r(B) + 1$, we deduce that 
$r(A) + r(B \cup f) = r(M) + 1$, so $M$ can be written as a $2$-sum with basepoint $p$ of two connected 2-polymatroids with ground sets $A \cup p$ and $B\cup f \cup p$. Let the former  be $M_1$. 

Suppose $M_1$ has at least three elements.
In the first case, by Lemma~\ref{new2}(\ref{new2i}), $M \ba q$ is the $2$-sum of the two connected $2$-polymatroids each with at least two elements, so by Proposition~\ref{connconn}, $M \ba q$ is connected. Now assume that $M_1 / q$ is connected. Then $\sqcap_{M_1} (\{q\}, \{p\}) = 0$ otherwise $p$ is a loop of $M_1 / q$, a contradiction. Hence $\sqcap_M (\{q\}, E_2) = 0$ as, by Proposition~\ref{dennis3.6},
\begin{equation*}
\begin{split}
\sqcap_M(\{q\}, E_2) &= r(\{q\}) + r(E_2) - r(E_2 \cup q) \\
&= r(\{q\}) + r(E_2) - r_{M_1}(\{p, q\}) - r(E_2) + 1 \\
&= r(\{q\}) + r(\{p\}) - r_{M_1}(\{p, q\}) \\
&= \sqcap_{M_1} (\{q\}, \{p\}).
\end{split}
\end{equation*}

Thus, by Lemma~\ref{new2}(\ref{new2ii}) and Proposition~\ref{connconn}, $M / q$ is connected since we again that the $2$-sum of two connected $2$-polymatroids with at least two elements.

As $q$ is in $A$ and hence in $F$,  both $M\ba q$ and $M/q$ have $N$ has a minor and so we obtain a contradiction.

We may now assume that $M_1$ consists of a single line $a$ through $p$. 

\begin{sublemma}\label{mb}
$M/a$ is connected. 
\end{sublemma}

Assume $M/a$ is disconnected. Then its ground set has a partition $(V,W)$ such that 
$r_{M/a}(V) + r_{M/a}(W) = r(M/a)$.  Now we may assume that $f$ is in $V$. Thus $W \subseteq B$ since $A = \{a\}$.   As $\sqcap_M(A,B) = 0$, it follows that  $r_{M/a}(W) = r_M(W)$. Hence $r_{M}(V\cup a) + r_{M}(W) = r(M)$; a contradiction. We conclude that \ref{mb} holds.

As $a \in F$, we know that $M/a$ has $N$ as a minor. Thus we have a contradiction that completes the proof of the theorem. 
\end{proof}

The argument above relies heavily on  the fact that we have a $2$-polymatroid. However, we believe that the main theorem also holds for $k$-polymatroids for all $k > 2$.

\begin{conjecture}
\label{k>2}
Let $M$ be a connected $k$-polymatroid and let $N$ be a connected minor of $M$. When $N \neq M$, there is an element $e$ of $E(M) - E(N)$ such that $M \ba e$ or $M / e$ is connected having $N$ as a minor.
\end{conjecture}

\section{Uniqueness}
\label{uniq}

By Theorem~\ref{mainone}, for every connected $2$-polymatroid $M$ and every connected proper minor $N$ of $M$, we can remove the elements of $E(M) - E(N)$ one at a time maintaining a connected $2$-polymatroid with $N$ as a minor. In this section, we consider what can be said about the uniqueness of this sequence of element removals

Now let  $M$ be a connected polymatroid and $N$ be a connected proper minor of $M$. An {\it admissible ordering} of $E(M) - E(N)$ is an ordering $(a_1,a_2,\dots,a_n)$ of the set $E(M) - E(N)$ such that, for each $k$ in $\{1,2,\dots,n\}$, there is a connected minor $M_k$ of $M$ with ground set $E(M) - \{a_1,a_2,\dots,a_k\}$ such that $M_k$ is a minor of $M_{k-1}$, where $M_0 = M$ and $M_n = N$. We give an example below to show that an admissible ordering may be unique. We shall show, however, that we always retain some flexibility with respect to the way in which the elements are removed unless $|E(M) - E(N)| = 1$. Formally, a {\it constrained admissible ordering}  is an ordering 
$((\alpha_1, a_1), (\alpha_2, a_2), \ldots, (\alpha_n, a_n))$
such that $E(M) - E(N) = \{a_1,a_2,\dots,a_n\}$ where each $\alpha_i$ is a deletion or contraction operation, and, for each $k$ in $\{1,2,\dots,n\}$, there is a connected minor $M_k$ of $M$ with ground set 
$E(M) - \{a_1,a_2,\dots,a_k\}$ where $M_k$ is obtained from $M_{k-1}$ by removing $a_k$ by the operation designated by $\alpha_k$, and $(M_0,M_n) = (M,N)$.

To construct a $2$-polymatroid with a unique admissible ordering, let $N$ be a simple non-empty connected matroid. Take $N \oplus U_{n,n}$ where the ground set of $U_{n,n}$ is $\{ b_0, b_1, \ldots, b_{n-1} \}$. Take $b_n \in E(N)$ and consider the $2$-polymatroid $M$ whose ground set is $E(N) \cup \{f_i : 1 \leq i \leq n \}$ where $f_i = \{b_{i-1}, b_i \}$ for all $i$, and the rank function of $M$ is induced by that of $N \oplus U_{n,n}$. Then $M$ is connected, $M \ba f_1, f_2, \ldots, f_k$ is connected for all $k$ in $\{1, 2, \ldots, n\}$, and $M \ba f_1, f_2, \ldots, f_n = N$. Thus $(f_1, f_2, \ldots, f_n)$ is an admissible ordering of $E(M) - E(N)$. It is not difficult to check that the admissible ordering is unique. Note, however, that $M \ba f_1 \ba f_2 = M / f_1 \ba f_2$, so this example does not give us a unique constrained admissible ordering. Indeed, as the next result shows, except in the trivial case, there can never be such a unique ordering.

\begin{theorem}
\label{uniqone}
Let $M$ be a connected $2$-polymatroid   and $N$ be a  connected proper minor of $M$. Then there is a unique constrained admissible ordering of $E(M) - E(N)$ if and only if  
$|E(M) - E(N)| = 1$.
\end{theorem}

The next two lemmas contain the core of the   proof of this theorem.

\begin{lemma}
\label{daggy}
Let each of $\dagger$ and $\ddagger$ denote a deletion or contraction operation. Suppose both $M\dagger e$ and $M\dagger e \ddagger f$ are connected, but $M\ddagger f$ is not. Then $\{e\}$ and $E(M) - \{e,f\}$ are the components of $M\ddagger f$. Moreover, 
$$M\ddagger f \ba e = M\ddagger f /e.$$
\end{lemma}

\begin{proof} Let $(X,Y)$ be a $1$-separating partition of $E(M\ddagger f)$ with $Y$ minimal and non-empty avoiding $e$. Then $\lambda_{M\ddagger f}(Y) = 0$. Thus, by Lemma~\ref{8.2.4}, $\lambda_{M\ddagger f \dagger e}(Y) = 0$. As $M\dagger e \ddagger f$ is  connected, $X - e = \emptyset$. Thus $\{e\}$ and $E(M) - \{e,f\}$ are  components of $M\ddagger f$. Hence, by Lemma~\ref{newt}, $M\ddagger f \ba e = M\ddagger f /e.$
\end{proof}

\begin{lemma}
\label{newt2}
Let $M$ be a connected  polymatroid   and $N$ be a  connected  minor of $M$. Let $e$ and $f$ be distinct elements of  $E(M)$  and let $Z = E(M) - \{e,f\}$. Suppose that $\{\alpha,\beta\} = \{\backslash,/\} = \{\gamma, \delta\}$. 
Assume that $M\alpha e$ and $M\alpha e\gamma f$ are connected.   Then either
 \begin{itemize}
\item[(i)] $M\gamma f$ is  connected and $M\alpha e\gamma f =  M\gamma f\alpha e$; or 
\item[(ii)] $M\gamma f$ is disconnected,  $M\alpha e\gamma f=  M\beta e\gamma f$, and
	\begin{itemize} 
	\item[(a)] $M\beta e$ is connected; or 
	\item[(b)] $M\beta e$ is disconnected, $M\alpha e\gamma f= M\delta f\beta e$, and $M\delta f$ is connected; or
	\item[(c)] $M\beta e$ and  $M\delta f$ are disconnected, and $M\alpha e\gamma f = M\alpha e\delta f$.
\end{itemize}
\end{itemize}
\end{lemma}

\begin{proof} We may assume that $M\gamma f$ is disconnected otherwise (i) holds.  As $M\alpha e\gamma f$ is connected,   Lemma~\ref{daggy} implies that  $\{e\}$ is a component of $M\gamma f$, and $M\gamma f\ba e =  M\gamma f/e$. Thus 
$$M\alpha e\gamma f = M\gamma f\alpha e = M\gamma f\beta e = M\beta e\gamma f.$$
We may assume  that $M\beta e$ is disconnected otherwise (ii)(a) holds. Then $M\beta e$  has $\{f\}$   as a component  and $M\beta e\ba f =  M\beta e/f$. Thus  
$$M\alpha e\gamma f = M\beta e\gamma f = M\beta e\delta f = M\delta f\beta e.$$

We may now assume  that $M\delta f$ is disconnected otherwise (ii)(b) holds. Then $M\delta f\ba e = M\delta f/e$. Thus $M\delta f\beta e = M\delta f\alpha e$. Hence  
 $M\alpha e\gamma f = M\delta f\beta e = M\delta f\alpha e =  M\alpha e\delta f$, and (ii)(c) holds.
\end{proof}

We are now able to prove the main result of this section. 

\begin{proof}[Proof of Theorem \ref{uniqone}]
We may assume that $|E(M) - E(N)| \ge 2$. 
Let  $((\alpha, e),(\gamma, f),\linebreak
(\alpha_3,a_3)\dots, (\alpha_k,a_k))$  be   a constrained admissible ordering of $E(M) - E(N)$. We use Lemma~\ref{newt2} to show that  $E(M) - E(N)$  has a constrained admissible ordering $((\alpha_1, a_1),(\alpha_2,a_2), 
(\alpha_3,a_3),\dots, (\alpha_k,a_k))$  in which $((\alpha, e),(\gamma, f)) \neq ((\alpha_1, a_1),(\alpha_2,a_2))$ and $\{e,f\} = \{a_1,a_2\}$. If $M\gamma f$ is connected, then we can take  $((\alpha_1, a_1),(\alpha_2,a_2))$ to be $((\gamma, f), (\alpha, e))$. Using the notation of Lemma~\ref{newt2}, if $M\gamma f$ is disconnected but $M \beta e$ is connected, then we can take $((\alpha_1, a_1),(\alpha_2,a_2))$ to be $((\beta, e), (\gamma, f))$. 
Now suppose that $M\gamma f$ and  $M \beta e$ are disconnected. If $M \delta f$ is connected, then we can take $((\alpha_1, a_1),(\alpha_2,a_2))$ to be $((\delta, f),   (\beta, e))$. Finally, if $M \delta f$ is disconnected, then we can take $((\alpha_1, a_1),(\alpha_2,a_2))$ to be $((\alpha, e),   (\delta, f))$. 
\end{proof}

\end{document}